\documentclass[12pt]{amsart}
\usepackage{amsmath,amssymb}
\usepackage[margin=1in]{geometry}

\theoremstyle{plain}
\newtheorem{thm}{Theorem}[section]
\newtheorem{counter}[thm]{Counterexample}
\newtheorem*{mainassump}{Main Assumption}

\newtheorem{remark}{Remark}  

\newtheorem{corollary}[thm]{Corollary}
\newtheorem{lemma}[thm]{Lemma}

\renewcommand{\div}{\operatorname{div}}
\newcommand{\trace}{\operatorname{tr}}

\newcommand{\D}{\displaystyle}
\newcommand{\ip}[2]{\ensuremath{\langle #1 , #2 \rangle}}
\newcommand{\tp}{\texttt{p}}
\renewcommand{\span}{\operatorname{span}}

\theoremstyle{definition}
\newtheorem{definition}{Definition}
\theoremstyle{remark}
\numberwithin{equation}{section}
\begin{document}
 
\title{Polarizing Anisotropic Heisenberg Groups}
\author{Thomas Bieske}
\address{ Department of Mathematics and Statistics,
University of South Florida, Tampa, FL 33620, USA}
\email{tbieske@mail.usf.edu}

\begin{abstract}
We expand the class of polarizable Carnot groups by implementing a technique to polarize anisotropic Heisenberg groups.  
\end{abstract}
\maketitle
\section{Background and Motivation}
In \cite{BT}, Balogh and Tyson establish the concept of polarizable Carnot groups. Polarizable Carnot groups are marked by the ability to create a system of polar coordinates that properly integrates with the sub-Riemannian environment. This produces consequences such as sharp constants for the Moser-Trudinger inequality, capacity formulas, and closed-form fundamental solutions to the $\tp$-Laplace equation. The only non-Euclidean examples are those in the class of groups of Heisenberg-type. (See Section 3 for further discussion.)  We extend the class of polarizable groups to include anisotropic Heisenberg groups. After a brief discussion of general Carnot groups in Section 2, groups of Heisenberg-type in Section 3, and polarizable Carnot groups in Section 4, we present our technique in Section 5.  

\section{Carnot Groups}
We begin by denoting an arbitrary Carnot group in $\mathbb{R}^N$ by $G$ and its
corresponding Lie Algebra by $g$.   Recall that $g$ is nilpotent and
stratified, resulting in the decomposition $$g= V_1 \oplus V_2 \oplus
\cdots \oplus V_k$$ for appropriate vector spaces that satisfy the Lie
bracket relation $[V_1,V_j]=V_{1+j}.$ We set $\dim V_i=n_i$ and denote
a basis for $g$ by $$X_{11},X_{12},\ldots ,
X_{1n_1},X_{21},X_{22},\ldots , X_{kn_k}$$ so that
\begin{equation*} 
V_i  =  \span\{X_{i1},X_{i2},\ldots ,X_{in_i}\}.
\end{equation*} 
The Lie Algebra $g$ is associated with the group $G$ via the exponential map $\exp: g \to G.$
For $X\in g$, we let $\Theta_X : \mathbb{R} \to G$ be the unique integral curve of $X$ with the following properties:  
\begin{eqnarray*}
\Theta ' (t)|_{t=0} & = & X \\
\Theta(0) & = & 0.
\end{eqnarray*}
We define the diffeomorphism $\exp:g \to G$ by $ \exp X= \Theta_X(1)$. Coordinates in $G$ arise from the image of the exponential map. That is, 
\[\exp \bigg(\sum_{i=1}^k\sum_{j=1}^{n_i}x_{ij}X_{ij}\bigg)=(a_{11},a_{12},\ldots ,a_{1n_1},a_{21},a_{22},\ldots , a_{kn_k})\in G.\] One can chose the basis so that $\exp$ is the identity, that is, so that in the relation above, $a_{ij}=x_{ij}$. 

The product of exponentials obeys the Baker-Campbell-Hausdorff formula (see, for example, \cite{BO:LAG}) 
\begin{equation}\label{BCH}
(\exp X)(\exp Y) = \exp ( X + Y + \frac{1}{2}[X,Y]+R(X,Y)) 
\end{equation}
where $R(X,Y)$ are terms formed by iterated brackets of $X$ and $Y$ of order at least 3.  
In particular, if for $X, Y\in g$, we have  $(a_{11},a_{12},\ldots ,a_{1n_1},a_{21},a_{22},\ldots , a_{kn_k})=\exp X$ and $
(b_{11},b_{12},\ldots ,b_{1n_1},b_{21},b_{22},\ldots , b_{kn_k})=\exp Y$, then Equation \eqref{BCH} induces a  (non-abelian) algebraic group law on $G$. The identity element of $G$ is denoted by $0$ and called the origin. 

Endowing $g$ with an inner product $\ip{\cdot}{\cdot}$
 and related norm $\|\cdot\|$, induces a natural metric on $G$, called the
Carnot-Carath\'{e}odory distance, defined for the points $p$ and $q$
as follows: 
 \begin{equation}\label{distance}
d_C(p,q)= \inf_{\Gamma} \int_{0}^{1} \| \gamma '(t) \| dt 
\end{equation}
where the set $ \Gamma $
is the set of all  curves $ \gamma $ such that $ \gamma (0) = p,
 \gamma (1) = q $ and $\gamma '(t) \in V_1$.  
 By Chow's theorem (see, for example,
\cite{BR:SRG}) any two points
 can be connected by such a curve, which means $ d_C(p,q) $ is an honest
metric.  We may define a Carnot-Carath\'{e}odory ball of radius $r$ centered at a
point $p_0$ by $$\mathfrak{B}(p_0,r)=\{p\in G : d_C(p,p_0) < r\}.$$

 \subsection{Calculus}\label{oncalc} 
When the basis is orthonormal, a smooth function $u: G \to \mathbb{R}$ has the horizontal derivative given by
$$\nabla_0 u=(X_{11}u,X_{12}u, \ldots, X_{1n_1}u )$$ and the symmetrized horizontal second derivative matrix, 
denoted by $(D^2u)^\star$, with entries
\begin{eqnarray*} 
((D^2u)^\star)_{ij} =  \frac{1}{2} (X_{1i}X_{1j}u+X_{1j}X_{1i}u) 
\end{eqnarray*} 
for $i,j=1,2,\ldots , n_1.$  
\begin{remark}
For notational purposes, we shall set $X_i=X_{1i}$. 
\end{remark}

We recall that for any open set $\mathcal{O} \subset G$, the
function $f$ is in the horizontal Sobolev space $W^{1,P}(\mathcal{O})$ if
$f$ and $ X_if$ are in $ L^\tp(\mathcal{O}) $ for $i=1,2,\ldots , n_1$. Replacing $ L^\tp(\mathcal{O})
$ by $
L_{loc}^\tp(\mathcal{O})$, the space $ W_{loc}^{1,\tp}(\mathcal{O}) $ is defined similarly.  The space
$W_{0}^{1,\tp}(\mathcal{O})$ is the closure in $W^{1,\tp}(\mathcal{O})$ of smooth functions with
compact support.  For more complete details on calculus on Carnot groups, see \cite {FS:HSHG} ,\cite{H:CCG}, \cite{HH:QR}, and \cite{St:HA}.

Using the above derivatives, we define the horizontal $\tp$-Laplacian of a smooth function $f$ for
$1<\tp<\infty$ by 
\begin{eqnarray*}
\Delta_{\tp}f & = & \div(\|\nabla_0 f\|^{\tp-2}\nabla_0f)= \sum_{i=1}^{n_1}X_i(\|\nabla_0 f\|^{\tp-2}\nabla_0f)\\
  & = & \|\nabla_0 f\|^{\tp-2} \trace((D^2f)^\star)+(\tp-2) \|\nabla_0 f\|^{\tp-4}
\ip{(D^2f)^\star\nabla_0 f}{\nabla_0 f}.
\end{eqnarray*}

Formally taking the limit as $\tp$ goes to infinity
results in the infinite Laplacian which is defined by
\begin{equation*}
 \Delta_{\infty}f  =  \sum_{i,j=1}^{n_1} X_ifX_jfX_iX_jf = \ip{(D^2f)^{\star}\nabla_0f}{\nabla_0f}. 
\end{equation*}

\section{Groups of Heisenberg type.}
Groups of Heisenberg type were first introduced in \cite{K:LGHT}. We begin by recalling that the center $Z$ of a Lie Algebra is defined by 
\[Z=\{z\in g: [v,z]=0\  \forall v\in g\}.\]
A Heisenberg-type group is then defined as follows: 
\begin{definition}\cite[Definition 18.1.1]{BLU}\label{Jdef}
A Heisenberg-type Lie Algebra  is a finite-dimensional real Lie algebra $g$ which can be endowed with an inner product $\ip{\cdot}{\cdot}$ such that $[Z^\perp, Z^\perp]=Z$
where $Z$ is the center of $g$ and moreover, for every fixed $z\in Z$, the map
$J_z:Z^\perp\to Z^\perp$ defined by $\ip{J_z(v)}{w} = \ip{z}{[v,w]}\  \forall w\in Z^\perp$ 
is an orthogonal map whenever $\ip{z}{z}=1$. 

A Heisenberg-type group (also called group of Heisenberg-type) is a Carnot group with a Heisenberg-type Lie Algebra. 
\end{definition}
\begin{remark}
By scaling, we have 
\begin{equation}
(J_z)^2=-\|z\|^2Id 
\end{equation}
\noindent for all $z$ in $Z$.  
\end{remark}
 
Given a group $G$ of Heisenberg-type with corresponding Lie Algebra $g$, we have 
$$g=Z^\perp \oplus Z$$ where $Z$ is the center of $g$. The exponential map then yields 
for all $x\in G$, $$x=\exp (v(x)+z(x))$$ for $v(x)\in Z^\perp$ and $z(x)\in Z$. We then define the following on $G$:
\begin{eqnarray*}
\mathcal{Q} & = & \dim Z^\perp + 2\cdot \dim Z \\
\rho(x) & = & (\|v(x)\|^4+16 \|z(x)\|^2)^{\frac{1}{4}} \\
B_r & = & \{x\in G: \rho(x)< r \} \\
|B_r|_\tp & = & \int_{B_r} \|\nabla_0\rho\|^\tp \\
\textmd{and\ \ } \omega_\tp & = & |B_1|_\tp
\end{eqnarray*}

In \cite {CDG:CE}, Capogna, Danielli, and Garofalo find a closed-form formulation for the fundamental solution to the $\tp$-Laplace equation in groups of Heisenberg-type via the following theorem. 
\begin{thm}\cite[Theorem 2.1]{CDG:CE}\label{pfundsoln}
Fix $1<p<\infty$. Define the constant $C_\tp$ by 
\begin{eqnarray*}
C_\tp =\left\{\begin{array}{cc}
 \frac{\tp-1}{\tp-\mathcal{Q}}(\mathcal{Q}\omega_\tp)^{-\frac{1}{\tp-1}} & \textmd{when\ \ \ } \tp \neq \mathcal{Q} \\
 \mbox{} & \mbox{} \\
(\mathcal{Q}\omega_\tp)^{-\frac{1}{\tp-1}} & \textmd{when\ \ \ } \tp \neq \mathcal{Q}
 \end{array}\right.
 \end{eqnarray*}
and the function $\Gamma_\tp(x)$ by 
\begin{eqnarray}\label{gammadef}
\Gamma_\tp(x) =\left\{\begin{array}{cc}
C_\tp \rho^{\frac{\tp-\mathcal{Q}}{\tp-1}} & \textmd{when\ \ \ } \tp \neq \mathcal{Q} \\
 \mbox{} & \mbox{} \\
C_\tp \log\rho & \textmd{when\ \ \ } \tp \neq \mathcal{Q}.
 \end{array}\right.
 \end{eqnarray}
Then $\Gamma_\tp(x)$ is a fundamental solution to the equation 
$$\Delta_\tp u=0$$
with singularity at the identity element $e\in G$.  A fundamental solution with singularity at any other point of $G$ is obtained by left-translation of $\Gamma_\tp$.
\end{thm}

This result motivated the following corollary:
\begin{corollary}\cite{B:Phd}\label{fundcor}
Let $G$ be a Heisenberg-type group. Then on $G\setminus\{e\}$, the function $\rho$ in the theorem above satisfies $$\Delta_{\infty}\rho=0.$$
\end{corollary}

\section{Polarizable Carnot Groups}
In \cite{BT}, Balogh and Tyson produce a procedure for constructing polar coordinates in certain Carnot groups, called polarizable Carnot groups. Polarizable Carnot groups are defined via the following definition: 
\begin{definition}\cite[Definition 2.12]{BT}\label{maindef}
We say that a Carnot group $G$ is polarizable if the homogeneous norm $N = u_2^{\frac{1}{2-\mathcal{Q}}}$ associated to Folland's [see \cite{F:SE}] solution $u_2$ for the 2-Laplacian 
$\Delta_2$ satisfies $\Delta_\infty N=0$ in $G\setminus\{0\}$.
\end{definition}

It is shown in \cite[Section 5]{BT} that groups of Heisenberg-type are polarizable. To date, these are the only examples of polarizable Carnot groups outside of $\mathbb{R}^n$. In particular, \cite[Section 6]{BT} shows that ``polarizable" is a fragile concept, unstable under small perturbations of the underlying Lie Algebra. The specific counterexample given is 
an anisotropic Heisenberg group in $\mathbb{R}^5$ generated by the vectors
\begin{eqnarray*}
X = \frac{\D \partial}{\D \partial x} + y  \frac{\D \partial}{\D \partial t}, & 
Y = \frac{\D\partial}{\D \partial y} - x  \frac{\D \partial}{\D \partial t}, \\
Z = \frac{\D \partial}{\D \partial z} + 2w  \frac{\D \partial}{\D \partial t}, & 
\textmd{and\ \ } \ \ W = \frac{\D \partial}{\D \partial w} - 2z \frac{\D \partial}{\D \partial t} 
\end{eqnarray*}
This counterexample can be generalized as in the following:
\begin{counter}\label{fakecounter}
For $j=1,2, \ldots, n$, let $L_j\in\mathbb{R}\setminus\{0\}$. Consider the following vector fields in $\mathbb{R}^{2n+1}$:
\begin{eqnarray}
X_j = \left\{\begin{array}{cl}
\frac{\D\partial}{\D\partial x_j} - L_j x_{j+n} \frac{\D\partial}{\D\partial t} & \textmd{for\ } j=1, 2, \ldots n \\
\mbox{} & \mbox{} \\
\frac{\D\partial}{\D\partial x_j} + L_{j-n} x_{j-n} \frac{\D\partial}{\D\partial t} & \textmd{for\ }  j=n+1, n+2, \ldots 2n 
\end{array}\right.
\end{eqnarray}
Resulting Lie brackets are given for $j<k$ by 
\begin{eqnarray*}
[X_j,X_k]=\left\{\begin{array}{cc}
2L_j\frac{\D\partial}{\D\partial t} & \textmd{for\ } k=n+j \\
\mbox{} & \mbox{} \\
0 & \textmd{for\ } k\neq n+j
\end{array}\right.
\end{eqnarray*}
We then add the vector $T=\frac{\D\partial}{\D\partial t}$ to form a basis for $\mathbb{R}^{2n+1}$ and stratify by $\mathbb{R}^{2n+1}=V_1\oplus V_2$ where 
\[V_1=\span\{X_1, X_2, \ldots, X_{2n}\} \textmd{\ \ and\ \ } V_2=\span\{T\}\]. 

We use this Lie Algebra and the exponential map of Section 2 to produce a step-two Carnot group. Standard calculations yield that that exponential map is the identity, that is, 
\[\exp \bigg(\sum_{j=1}^{2n}x_{j}X_{j}+tT\bigg)=(x_{1},x_{2},\ldots ,x_{2n},t).\] Consequently the algebraic group law is given by 
\begin{eqnarray*}
\lefteqn{(x_{1},x_{2},\ldots ,x_{2n},t)*(y_{1},y_{2},\ldots ,y_{2n},s)} \\
& & \bigg(x_{1}+y_{1} ,x_{2}+y_{2},\ldots ,x_{2n}+y_{2n},t+s+2\sum_{j=1}^nL_j(x_jy_{j+n}-y_jx_{j+n})\bigg). 
\end{eqnarray*}
\end{counter}
\qed

We have the following theorem.
\begin{thm}\label{fundsoln}
Let $n=2$ in Counterexample \ref{fakecounter} and let $L_2=2L_1$. Set  
\[N = \frac{(B^2 + t^2)^{\frac{1}{8}}(A-B + \sqrt{B^2 + t^2})^{\frac{3}{8}}}{(B + \sqrt{B^2 + t^2})^{\frac{1}{8}}}\] where 
\[B=B(x_1,x_2,x_3,x_4)=|L_1| \bigg(\frac{1}{2} x_1^2 +  x_2^2 + \frac{1}{2}  x_3^2 + x_4^2\bigg)\]
and
\[A=A(x_1,x_2,x_3,x_4)=|L_1| (x_1^2 +  x_2^2 + x_3^2 + x_4^2).\]
Then, for an appropriate constant $C$, we have $CN^{-4}$ is the fundamental solution to $\Delta_2 u=0$ but $\Delta_\infty N\neq 0$. 
\end{thm} 
\begin{proof}
The theorem was proved in \cite[Section 6]{BT} for the case $L_1=1$ by using the Beals-Gaveau-Greiner \cite{BGG} formula for the fundamental solution. In the general case, the computations are similar and omitted.  
\end{proof}
Definition \ref{maindef} produces the following corollary.   
\begin{corollary}\label{cone}
Counterexample \ref{fakecounter} need not be polarizable.  
\end{corollary}
Theorem \ref{pfundsoln} and the discussion after Definition \ref{maindef} produce the following corollary.
\begin{corollary}\label{ctwo}
Counterexample \ref{fakecounter} need not be a group of Heisenberg-type.  
\end{corollary}
\section{Counterexample Revisited}
In this section, we will take a closer look at Counterexample \ref{fakecounter} and show that while the theorem is true, we can produce a procedure to polarize these Carnot groups and, in effect, falsify the Corollary \ref{cone}. 

We begin by examining the underlying assumptions on the Beals-Gaveau-Greiner \cite{BGG} formula Subsection \ref{oncalc} . In particular, we assumed in this formula, and consequently Theorem \ref{fundsoln}, that the vector fields are orthonormal. We will alter this assumption by operating under the following main assumption:

\begin{mainassump}
The vector fields in Counterexample \ref{fakecounter} satisfy the following:
\begin{eqnarray*}
\|X_j\|^2=\ip{X_j}{X_j} & = &2 |L_j| \textmd{\ for\ } j=1, 2, \ldots, n\\
\|X_j\|^2=\ip{X_j}{X_j} & = & 2|L_{j-n}| \textmd{\ for\ } j=n+1, n+2, \ldots, 2n  \\
\ip{X_j}{X_k} & = & 0 \textmd{\ for\ } j\neq k  \\
\ip{X_j}{T} & = & 0 \\
\textmd{and\ }\ \ip{T}{T} & = & 1. 
\end{eqnarray*}
In particular, the basis is orthogonal but no longer orthonormal. 
\end{mainassump}
\begin{remark}
We note that these assumptions do not change the center $Z$ of the Lie Algebra, the exponential map or the algebraic group law. However, the metric space properties have been altered, as Equation \eqref{distance} relies on this norm. 
\end{remark}
Under the Main Assumptions, we consider the map $J_T:V_1\to V_1$ from Definition \ref{Jdef}. By construction of the Lie Algebra, we have  
\begin{eqnarray*}
\ip{T}{[X_j,X_k]}=\left\{\begin{array}{cl}
\ip{T}{2L_jT}=2Lj & \textmd{when\ \ } j=1,2,\ldots, n \ \textmd{and\ \ } k=j+n \\
\mbox{} & \mbox{} \\
-\ip{T}{2L_jT}=-2Lj & \textmd{when\ \ } j=n+1,n+2,\ldots, 2n \ \textmd{and\ \ } k=j-n \\
\mbox{} & \mbox{} \\
0 & \textmd{otherwise.}
\end{array}\right.
\end{eqnarray*}
We conclude that the matrix $\mathcal{J}$ of the map $J_T$ is given by 
\[\begin{bmatrix}
0_{n\times n} & -I_{n\times n} \\
\mbox{} & \mbox{} \\
I_{n\times n} & 0_{n\times n}
\end{bmatrix}\]
and thus by Definition \ref{Jdef}, the resulting Carnot group is a group of Heisenberg-type. 

\subsection{Calculus using the Orthogonal Vector Field}
Because we are not employing orthonormal vectors, we must modify the formula for  divergence, and hence the formula for the $\Delta_\tp$ Laplace operator. This is detailed in the following lemma, whose proof results from the formulas for gradient and divergence in curvilinear coordinates (\cite[Chapter 2]{book1} or \cite[Chapter 7]{book2}).
\begin{lemma}\label{pLapdef}
Consider a vector field $F=\sum_{i=1}^{2n}g_iX_i$, where the vector fields $X_i$ satisfy our Main Assumption.  The divergence formula produces
\[\div F= \sum_{i=1}^{n}\frac{1}{\sqrt{2|L_i|}}X_ig_i+\sum_{i=n+1}^{2n}\frac{1}{\sqrt{2|L_{i-n}|}}X_ig_i.
\]
Now, let $f(x_1,x_2,\ldots, x_{2n}, t)$ be a smooth function and define the operator $\mathcal{M}(f)$ by
\[\mathcal{M}(f)=  \bigg(\sum_{j=1}^{n}\frac{1}{2|L_j|}(X_jf)^2+\sum_{j=n+1}^{2n}\frac{1}{2|L_{j-n}|}(X_jf)^2\bigg)^{\frac{1}{2}}.\]
Then, for $1<\tp< \infty$, the $\tp$-Laplace operator is given by 
\begin{eqnarray*}
\lefteqn{\Delta_\tp f = \sum_{i=1}^{n}\mathcal{M}(f)^{\tp-2}\frac{1}{2|L_i|}X_iX_if+
\sum_{j=n+1}^{2n}\mathcal{M}(f)^{\tp-2}\frac{1}{2|L_{i-n}|}X_iX_if+ (\tp-2)\mathcal{M}(f)^{\tp-4}} && \\
 & & \mbox{}\times\Bigg(\sum_{i=1}^{n}\bigg(\sum_{j=1}^n\frac{1}{4|L_iL_{j}|}X_iX_jfX_ifX_jf+\sum_{j=n+1}^{2n}\frac{1}{4|L_iL_{j-n}|}X_iX_jfX_ifX_jf\bigg)\\
 & & \mbox{}+\sum_{i=n+1}^{2n}\bigg(\sum_{j=1}^n\frac{1}{4|L_jL_{i-n}|}X_iX_jfX_ifX_jf+\sum_{j=n+1}^{2n}\frac{1}{4|L_{i-n}L_{j-n}|}X_iX_jfX_ifX_jf\bigg)\Bigg)
\end{eqnarray*}
and the $\infty$-Laplace operator is given by 
\begin{eqnarray*}
\lefteqn{\Delta_\infty f   =  \sum_{i=1}^{n}\bigg(\sum_{j=1}^n\frac{1}{4|L_iL_{j}|}X_iX_jfX_ifX_jf+\sum_{j=n+1}^{2n}\frac{1}{4|L_iL_{j-n}|}X_iX_jfX_ifX_jf\bigg)}&&\\
& & \mbox{} + \sum_{i=n+1}^{2n}\bigg(\sum_{j=1}^n\frac{1}{4|L_jL_{i-n}|}X_iX_jfX_ifX_jf+\sum_{j=n+1}^{2n}\frac{1}{4|L_{i-n}L_{j-n}|}X_iX_jfX_ifX_jf\bigg).
\end{eqnarray*}
\end{lemma}

We then invoke Theorem \ref{pfundsoln} and Corollary \ref{fundcor} to establish the formula for  the fundamental solution to the $\tp$-Laplace equation. In particular,  
\begin{thm}
Let
\[\rho(x_1,x_2,\ldots, x_n,t)= \Bigg(\bigg(\sum_{i=1}^{n}2|L_i|x_i^2+\sum_{i=n+1}^{2n}2|L_{i-n}|x_i^2\bigg)^2+16t^2\Bigg)^{\frac{1}{4}}.\]
Then the function $\Gamma_\tp(x)$ defined in Equation \eqref{gammadef} is the fundamental solution to the $\tp$-Laplace equation for $1<\tp<\infty$ and on $G\setminus\{e\}$, we have $\Delta_{\infty}\rho=0.$
\end{thm}
We now may invoke \cite[Section 3]{BT} to construct polar coordinates.


\begin{thebibliography}{99}
\bibitem{book1}Arfken, George B.; Weber, Hans J. \emph{Mathematical Methods for Physicists}; Sixth Edition, Elsevier Academic Press: Burlington, MA. 2005.
\bibitem{BT} Balogh, Zolt\'{a}n.; Tyson, Jeremy. Polar Coordinates in Carnot Groups. Math Z. \emph{2002}, 241:4, 697--730.
\bibitem{BGG}Beals, R.; Gaveau, B.; Greiner, P. The Green function of model step two hypoelliptic operators and the analysis of certain tangential Cauchy Riemann complexes. Adv. Math. \emph{1996}, 121, 288--345.
\bibitem{BR:SRG}Bella\" {\i}che, Andr\' {e}. The Tangent Space in
Sub-Riemannian Geometry. In \emph{Sub-Riemannian Geometry};
Bella\" {\i}che, Andr\' {e}., Risler, Jean-Jacques., Eds.; 
Progress in Mathematics;  Birkh\" {a}user: Basel, Switzerland.
1996; Vol. 144, 1--78.
\bibitem{B:Phd}Bieske, Thomas. Lipschitz Extensions in the Heisenberg Group. Ph.D. thesis, University of Pittsburgh, \emph{1999}.
\bibitem{BLU}Bonfiglioli, A.; Lanconelli, E.; Uguzzoni, F. \emph{Stratified Lie Groups and Potential Theory for their Sub-Laplacians}; Springer-Verlag:Berlin. 2007.
\bibitem{BO:LAG}Bourbaki, Nicolas, \emph{Lie Groups and Lie Algebras, Chapters 1--3,} Elements of Mathematics, Springer-Verlag, 1989.
\bibitem{CDG:CE}Capogna, Luca.; Danielli, Donatella.; Garofalo, Nicola. Capacitary Estimates and the Local Behavior of Solutions of Nonlinear Subelliptic Equations. Amer. J. of Math. \emph{1996}, 118:6, 1153--1196.
\bibitem{F:SE}Folland, G.B.
Subelliptic Estimates and Function Spaces on Nilpotent Lie Groups.
Ark. Mat. \emph{1975}, 13, 161--207.
\bibitem{FS:HSHG}Folland, G.B.; Stein, Elias M.
\emph{Hardy Spaces on Homogeneous Groups};
Princeton University Press: Princeton, NJ. 1982.
\bibitem{H:CCG}Heinonen, Juha.
Calculus on Carnot Groups.
In \emph{Fall School in Analysis Report No. 68},
Fall School in Analysis, Jyv\"{a}skyl\"{a}, 1994;  
Univ. Jyv\"{a}skyl\"{a}: Jyv\"{a}skyl\"{a}, Finland. 1995; 1--31. 
\bibitem{HH:QR}Heinonen, Juha.; Holopainen, Ilkka.
Quasiregular Maps on Carnot Groups.
J. of Geo. Anal. \emph{1997},  7:1, 109--148.
\bibitem{K:LGHT}Kaplan, Aroldo. Lie Groups of Heisenberg Type.
Rend. Sem. Mat. Univ. Politec. Torino 1983 Special Issue \emph{1984}, 117--130.
\bibitem{book2}Spiegel, Murray R. \emph{Vector Analysis and an Introduction to Tensor Analysis}; Schaum's Outline Series, McGraw-Hill: New York, NY. 1959. 
\bibitem{St:HA}Stein, Elias M.  \emph{Harmonic analysis: real-variable methods, orthogonality, and oscillatory integrals.} With the assistance of Timothy S. Murphy. Princeton Mathematical Series, 43. Monographs in Harmonic Analysis, III. Princeton University Press, Princeton, NJ, 1993.
\end{thebibliography}
\end{document}